\documentclass[reqno, final, a4paper]{amsart}
\pdftrailerid{}
\usepackage{color}
\usepackage{amsmath, amssymb, amsthm}
\usepackage{mathrsfs}
\usepackage{mathtools}
\usepackage{fullpage}
\usepackage{setspace}
\usepackage[noadjust]{cite}
\usepackage{graphics}
\usepackage{pifont}
\usepackage{cancel}
\usepackage{tikz}
\usepackage{bbm}
\usepackage[T1]{fontenc}
\usetikzlibrary{arrows.meta}

\usepackage{environ}
\usepackage{paralist}
\usepackage{url}
\usepackage[linesnumbered,ruled,vlined]{algorithm2e}
\usepackage[noend]{algpseudocode}
\usepackage[labelfont=bf]{caption}
\usepackage{framed}
\usepackage[framemethod=tikz]{mdframed}
\usepackage{appendix}
\usepackage[textsize=tiny]{todonotes}
\usepackage{tcolorbox}
\usepackage[shortlabels]{enumitem}

\usepackage{euflag}
\usepackage{cases}

\usepackage[mathlines]{lineno}
\usepackage{etoolbox} %

\newcommand*\linenomathpatch[1]{%
   \expandafter\pretocmd\csname #1\endcsname {\linenomath}{}{}%
   \expandafter\pretocmd\csname #1*\endcsname{\linenomath}{}{}%
   \expandafter\apptocmd\csname end#1\endcsname {\endlinenomath}{}{}%
   \expandafter\apptocmd\csname end#1*\endcsname{\endlinenomath}{}{}%
 }
\newcommand*\linenomathpatchAMS[1]{%
    \expandafter\pretocmd\csname #1\endcsname {\linenomathAMS}{}{}%
    \expandafter\pretocmd\csname #1*\endcsname{\linenomathAMS}{}{}%
    \expandafter\apptocmd\csname end#1\endcsname {\endlinenomath}{}{}%
    \expandafter\apptocmd\csname end#1*\endcsname{\endlinenomath}{}{}%
}

\expandafter\ifx\linenomath\linenomathWithnumbers
\let\linenomathAMS\linenomathWithnumbers
\patchcmd\linenomathAMS{\advance\postdisplaypenalty\linenopenalty}{}{}{}
\else
\let\linenomathAMS\linenomathNonumbers
\fi

\linenomathpatchAMS{gather}
\linenomathpatchAMS{multline}
\linenomathpatchAMS{align}
\linenomathpatchAMS{alignat}
\linenomathpatchAMS{flalign}
\linenomathpatch{equation}
\linenomathpatch{numcases}
\linenomathpatch{figure}

\newtheorem{theorem}             {Theorem}[section]
\newtheorem{lemma}   [theorem]   {Lemma}

\newtheorem{proposition}  [theorem] {Proposition}   
\newtheorem{corollary}  [theorem] {Corollary}

\newtheorem{observation}[theorem] {Observation}   

\newtheorem{proto-theo}[theorem] {Proto-Theorem}

\def\Pr{\mathop{\mathbb{P}}\nolimits}

\let\:=\colon
\let\phi=\varphi

\newcommand{\cC}{\mathcal{C}}

\newcommand{\cF}{\mathcal{F}}

\newcommand{\cH}{\mathcal{H}}
\newcommand{\cI}{\mathcal{I}}

\newcommand{\cP}{\mathcal{P}}

\newcommand{\cR}{\mathcal{R}}
\newcommand{\cS}{\mathcal{S}}

\newcommand{\cW}{\mathcal{W}}

\newcommand{\cZ}{\mathcal{Z}}

\def\bbn{{\mathbb N}}
\def\NN{{\mathbb N}}

\def\cankvdW{\textrm{\rm can-$k$-vdW}}
\def\rkvdW{\mathop{\textrm{\rm $(r,k)$-vdW}}\nolimits}
\def\aksz{\mathop{\textrm{\rm $(\alpha,k)$-Sz}}\nolimits}
\def\akrb{\mathop{\textrm{\rm $(\alpha,k)$-rb}}\nolimits}
\def\hkap{\cH_{k\text{{\rm-AP}}}}

\let\epsilon\varepsilon
\let\eps\epsilon

\usepackage[colorlinks=true,allcolors=blue
]{hyperref}
\usepackage[noabbrev,capitalize,nameinlink]{cleveref}
\crefname{equation}{}{}
\crefformat{enumi}{#2#1#3}
\crefrangeformat{enumi}{#3#1#4 to~#5#2#6}
\crefmultiformat{enumi}{#2#1#3}
{ and~#2#1#3}{, #2#1#3}{ and~#2#1#3}

\title[A sparse canonical van der Waerden theorem]%
 {A sparse canonical van der Waerden theorem}

\author[Alvarado]{José~D.~Alvarado}

\address[J.\,D. Alvarado]{Faculty of Mathematics and Physics, University of Ljubljana, Slovenia}

\email{jose.alvarado@fmf.uni-lj.si}

\author[Kohayakawa]{Yoshiharu Kohayakawa}

\address[Y. Kohayakawa and G.\,O. Mota]{Instituto de Matem\'atica e Estat\'{\i}stica \\Universidade de 
	S\~ao Paulo, Rua do Mat\~ao 1010\\05508--090~S\~ao Paulo, Brazil}
\email{\{\,yoshi\,|\,mota\,\}@ime.usp.br}

\author[Morris]{Patrick Morris}
\address[P. Morris and M. Ortega]{Departament de Matem\`atiques, Universitat Polit\`ecnica de Catalunya (UPC), Carrer de Pau Gargallo 14, 08028  Barcelona, Spain}
\email{pmorrismaths@gmail.com, miquel.ortega.sanchez-colomer@upc.edu}

\author[Mota]{Guilherme~O.~Mota}

\author[Ortega]{Miquel Ortega}

\thanks{J. D. Alvarado was partially supported by FAPESP
  (2020/10796-0) and the European Union (ERC, KARST, project number
  101071836). Y. Kohayakawa was partially supported by FAPESP (2023/03167-5) and
  CNPq (407970/2023-1, 420838/2025-2, 315258/2023-3).
  P. Morris was supported by the Deutsche Forschungsgemeinschaft (DFG, German
Research Foundation) Walter Benjamin programme - project number 504502205 and by the European Union's Horizon Europe   Marie Sk{\l}odowska-Curie grant RAND-COMB-DESIGN - project number
101106032 {\euflag}. G. O. Mota was partially supported by CNPq
		(315916/2023-0,  420838/2025-2) and FAPESP
                (2023/03167-5, 2024/13859-4). M. Ortega was supported by the project PID2023-147202NB-I00, funded by MICIU/AEI/10.13039/501100011033/, as well as the FPI grant PRE2021-099120. This study was financed in part by CAPES, Brazil, Finance
		Code~001.}

\date{\today}	
	
\begin{document}
\onehalfspace

\begin{abstract}
 The canonical van der Waerden theorem asserts that, for sufficiently large~$n$, every colouring of $[n]$ contains either a monochromatic or a rainbow arithmetic progression of length $k$ ($k$-AP, for short). In this paper, we determine the threshold at which the binomial random subset $[n]_p$ almost surely inherits this canonical Ramsey type property. As an application, we show the existence of sets $A\subseteq [n]$ such that the $k$-APs in $A$ define a $k$-uniform hypergraph of arbitrarily high girth and yet any colouring of $A$ induces a monochromatic or rainbow $k$-AP. 
\end{abstract}

\maketitle

\section{Introduction}

A $k$-term arithmetic progression ($k$-AP for short) in the integers is a sequence $a,a+d,\ldots,a+(k-1)d$ for some $a,\,d\in \NN$ (in particular $d\neq 0$). One of the first and most celebrated results in  Ramsey theory is van der Waerden's theorem \cite{1927.vdW}, which states that for any $r,k \in \NN$, if $n\in \NN$ is sufficiently large, then  any $r$-colouring of $[n]$ results in a monochromatic $k$-AP. Erd\H{o}s and Graham \cite{1980.EG} explored what happens when one removes the restriction of using a bounded number of colours. They proved the beautiful \textit{canonical van der Waerden theorem} which states that for $k\in \NN$, if $n$ is sufficiently large then \textit{any} colouring of $[n]$ results in a $k$-AP which is \textsl{either} monochromatic   \textsl{or} \emph{rainbow}, with each element in the $k$-AP having a unique colour. 

In this paper, we prove a sparse random version of the canonical van der Waerden theorem.  The binomial random set $[n]_p$ is the set obtained by keeping each integer $x\in[n]$ independently with probability $p=p(n)\in [0,1]$. 
We say that a set $A\subseteq \NN$ is $\cankvdW$ if it has the canonical van der Waerden property, that is, any colouring of $A$ gives either a monochromatic or a rainbow $k$-AP. Our main theorem determines the threshold for $[n]_p$ to have this   property.

\begin{theorem}  \label{thm:random_canonical_vdw}
  Let $k \geq 3$ be an integer. Then there exist $c,\,C>0$ such that 
  \setcounter{equation}{-1}
  \begin{linenomath}
 \begin{numcases}
 {\lim_{n \to \infty} \Pr([n]_p  \mbox{ is } \cankvdW) =}
	  \label{can random 0} 0 &if $p < c n^{-1/(k-1)}$;\\ 
	 \label{can random 1}    1  &if $p > C n^{-1/(k-1)}.$ 
 \end{numcases}   
    \end{linenomath}
   \end{theorem}
As we shall soon see, the \ref{can random 0}-statement\footnote{The
  $0$-statement is the assertion that the limiting probability is~$0$ if~$p$ is as
in~\eqref{can random 0}.  We define the $1$-statement analogously.}
 is in fact immediate from previous results and our contribution is to give the matching \ref{can random 1}-statement and thus establish the threshold. 

\subsection{Sparse van der Waerden theorems}
\label{subsection:sparse_vdw}
To put Theorem \ref{thm:random_canonical_vdw} in context, we return to the original setting of van der Waerden's theorem. We say that a set $A\subseteq \NN$ is  $\rkvdW$  if any $r$-colouring of $A$ gives a monochromatic $k$-AP. Whilst van der Waerden's theorem tells us that there is some $W=W(r,k)\in \NN$ such that $[n]$ is $\rkvdW$ for any $n\geq W$, in the 1970s researchers began to question whether there are sparser sets $A\subseteq [n]$ which are also $\rkvdW$. If $A$ contains some arithmetic progression $S$ of length $W$, then certainly $A$ will be $\rkvdW$ just by considering the colouring on $S$. Therefore a first test for finding interesting sparse constructions that are $\rkvdW$ would be to avoid large arithmetic progressions in $A$.
Taking this to the extreme, Erd\H{o}s \cite{1974.E} asked if for all $r,k\in \NN$, there is some set $A$ which contains no $(k+1)$-AP and is still $\rkvdW$. These were then shown to exist independently by Spencer \cite{1975.Spencer} and Ne\v{s}et\v{r}il and R\"odl \cite{1976.NR}. %

To push further the sparsity condition on $A\subseteq [n]$, we consider $\hkap(A)$ which we define to be  the $k$-uniform hypergraph with vertex set $A$
 whose edges are the $k$-APs contained in $A$. Intuitively, in order for a set
 $A$ to be $\rkvdW$, you need that the $k$-APs in $A$ intersect significantly.
 On the other hand, one can enforce local sparsity of the $k$-APs in $A$ by
 forbidding short cycles in $\hkap(A)$.
 Here, a cycle~$\cC$ of length $2\leq\ell\in \NN$  in a hypergraph is
 a subgraph with distinct vertices $v_1,\ldots,v_\ell\in V(\cC)$ and
 an ordering of $E(\cC)$ as $e_1,\ldots,e_\ell$ such that $v_i\in
 e_i\cap e_{i+1}$ for $i=1,\ldots,\ell=|E(
 \cC)|$ (with $e_{\ell+1}=e_1$
 here). 
 The \emph{girth} of a hypergraph $\cH$ is the
    minimum length of a cycle that is a subgraph of~$\cH$. 
 Spencer \cite{1975.Spencer} conjectured that for any $r,k,g\in \NN$, there are
 sets $A\subseteq \NN$ that are $\rkvdW$ such that the girth of $\hkap(A)$ is at
 least~$g$. Note that $(k+1)$-APs in $A$ induce cycles of length two in $\hkap(A)$ and so Spencer's conjecture is a broad strengthening of the results mentioned above. The solution to the conjecture took  some time, eventually being proven in 1990  by R\"odl \cite{1990.R} (see also Ne\v{s}et\v{r}il and R\"odl \cite{1990.NR}).

Whilst the constructions giving sparse $\rkvdW$ sets listed above are all constructive, for the analogous problem in graphs, some early progress \cite{1986.FR} was made using random graphs. 
This indicated the use of random structures to give interesting examples of sparse constructions with Ramsey type properties and quickly led to the systematic study \cite{1992.LRV} of thresholds for these properties in random environments. In seminal work, R\"odl and Ruci\'nski \cite{1995.RR} determined many of these thresholds. In particular, answering a question of Lefmann and independently Erd\H{o}s and S\'os, they proved the following.

\begin{theorem}{\cite{1995.RR}}    \label{thm:random_vdw}
  Let $k\geq 3$ and $r\geq 2$ be integers. Then there exist $c,\,C>0$ such that 
  \setcounter{equation}{-1}
 \begin{numcases}
 {\lim_{n \to \infty} \Pr([n]_p  \mbox{ is } \rkvdW) =}
	  \label{random 0} 0 &if $p < c n^{-1/(k-1)}$;\\ 
	 \label{random 1}    1  &if $p > C n^{-1/(k-1)}.$ 
 \end{numcases}   
   \end{theorem}

At the location $n^{-1/(k-1)}$ of the threshold, one expects the
number of $k$-APs to be of the same order of magnitude as the number
of elements in $[n]_p$. Thus, on average, each element of $[n]_p$ lies
in constantly many $k$-APs. The \ref{random 1}-statement of Theorem~\ref{thm:random_vdw} shows that as soon as this average is large
enough, asymptotically almost surely (a.a.s.\ for short) $[n]_p$ will be $\rkvdW$. Thus Theorem \ref{thm:random_vdw} gives very sparse $\rkvdW$ sets. In fact using Theorem \ref{thm:random_vdw}, R\"odl and Ruci\'nski could reprove probabilistically the sparse van der Waerden results listed above.

\subsection{Sparse canonical van der Waerden theorems} Our main theorem, Theorem \ref{thm:random_canonical_vdw}, shows that the location of the threshold for the canonical van der Waerden property coincides with the threshold for the $r$-colour van der Waerden property. In fact, the \ref{can random 0}-statement of Theorem \ref{thm:random_canonical_vdw} follows easily from the \ref{random 0}-statement of Theorem \ref{thm:random_vdw}. Indeed for any $3\leq k\in \NN$ there is some $c>0$ such that if  $p < c n^{-1/(k-1)}$, then a.a.s.\ by Theorem \ref{thm:random_vdw} there is some 2-colouring of $[n]_p$ such that there are no monochromatic $k$-APs. Such a colouring also avoids rainbow $k$-APs simply because there are only 2 colours used. 
As in the $r$-colour van der Waerden case, we can also use Theorem \ref{thm:random_canonical_vdw} to derive the existence of locally sparse sets that have the canonical van der Waerden property. 

\begin{corollary} \label{cor:sparse_vdw}
    For any integers $3\leq k,\,g\in \NN$, there is some set $A\subseteq \NN$
    which is $\cankvdW$ such that $\hkap(A)$ has girth at least $g$. 
\end{corollary}

 As $(k+1)$-APs in $A$ give cycles of length 2 in $\hkap(A)$, Corollary \ref{cor:sparse_vdw} in particular gives sets that are $\cankvdW$ whilst containing no $(k+1)$-AP, the existence of which was previously proven by Pr\"omel and Rothschild \cite{1987.PR}. Our proof  of Corollary \ref{cor:sparse_vdw} follows a scheme used by R\"odl and Ruci\'nski \cite{1995.RR}.

\subsection{The development of random Ramsey theory}
The study of thresholds for Ramsey properties in discrete random
structures started with the pioneering works of Frankl and R\"odl
\cite{1986.FR} exploring sparse Ramsey graphs, {\L}uczak, Ruci\'nski and Voigt \cite{1992.LRV} who initiated a systematic study  and R\"odl and Ruci\'nski \cite{1995.RR} who determined many of the thresholds, also in arithmetic settings. Since then, the area has become a cornerstone   of modern probabilistic combinatorics, with many exciting developments both in methods and results. Around 10 years ago, there was an important breakthrough which placed Ramsey thresholds in random structures into a more general  framework. Simultaneously and independently, Conlon and Gowers \cite{2016.CG} and Schacht \cite{2016.Schacht} developed powerful \textit{transference principles} giving $1$-statements  for thresholds for a host of different monotone properties (like in Theorem \ref{thm:random_vdw}), solving several longstanding open problems in the process. Shortly after this, Balogh, Morris and Samotij \cite{2015.BMS} and Saxton and Thomason \cite{2015.ST} developed the theory of \textit{hypergraph containers}, giving  lemmas that describe the structure of the family of independent sets in well-distributed hypergraphs. It turns out that this level of abstraction leads to a stunning level of applicability. Among  many other applications, hypergraph containers provide relatively simple proofs to many of the previous theorems \cite{2016.CG,2016.Schacht} establishing thresholds for properties of discrete random structures. In particular, in the context of Ramsey properties, Nenadov and Steger \cite{2016.NS} showed how hypergraph containers can give beautiful proofs for $1$-statements of  random Ramsey theorems such as Theorem \ref{thm:random_vdw}.

These tools (and further beautiful ideas) have opened up avenues for
new directions in random Ramsey theory and the area has flourished in
recent years. In particular, we mention recent breakthroughs in
proving sharp thresholds for random Ramsey properties \cite{2022.FKSS}
and establishing thresholds for asymmetric Ramsey problems in graphs
\cite{2025.CMSW,2020.MNS}. Exploring canonical Ramsey theorems has
also only recently been addressed.  
Clearly, when studying canonical properties, one has to deal with the
fact that one allows an unbounded number of colours, and that adds an
extra layer of difficulty. 
Nevertheless, Kam\v{c}ev and Schacht \cite{2023.KS} succeeded in obtaining a
remarkable result establishing the location of the threshold for the
random graph to have the canonical Ramsey property with respect to
cliques, using the transference principle of Conlon and Gowers
\cite{2016.CG}. Alvarado, Kohayakawa, Morris and Mota also recently
proved a general theorem for canonical Ramsey properties in graphs
when the colourings are restricted to pre-assigned lists
\cite{2023.AKMM} and used this to give a random canonical Ramsey
theorem for even cycles \cite{2024.AKMM}.  These latter results, as
well as our proof of Theorem \ref{thm:random_canonical_vdw}, use
hypergraph containers as a key tool. We refer to the papers
\cite{2023.AKMM,2024.AKMM,2023.KS} for details of the results and in
particular for the definition of canonical Ramsey properties in
graphs, which has extra \textit{lexicographic}-type colour patterns
that do not appear in the $k$-AP setting.

We finish this introduction by remarking that it would be very interesting to develop sparse and/or random analogues to other canonical Ramsey theorems from additive combinatorics such as the canonical Rado theorem \cite{1986.L} of Lefmann or the recently proven canonical version of the polynomial van der Waerden theorem due to Gir\~ao \cite{2020.G} and independently by Fox, Wigderson and Zhao \cite{2020.FWZ}.

\subsection{Organisation}
In Section \ref{sec:prelims} we introduce some notation and discuss auxiliary results and lemmas we need in
the proof of Theorem~\ref{thm:random_canonical_vdw}, which is then proven in Section~\ref{sec:outline}. In Section \ref{sec:cor} we then prove Corollary \ref{cor:sparse_vdw}.

\subsubsection*{Acknowledgements} We are grateful to  the anonymous referee for their careful work.

\section{Preliminaries} \label{sec:prelims}

The binomial random set, denoted $[n]_p$, refers to the set obtained by taking every element of $[n]$
independently with probability $p = p(n)$.  

\subsection{Colourings} For $A\subseteq \NN$ and $r\in \NN$ we say that  a colouring $\chi:A\rightarrow \NN$ is an $r$-colouring if $|\chi(A)|\leq r$, where $\chi(A):=\{\chi(a):a\in A\}$. In such a case we will often identify $\chi(A)$ as $[r]$. For $\alpha>0$, we say a colouring $\chi:A\rightarrow \NN$ is 
\emph{$\alpha$-bounded} if $|\chi^{-1}(i)|\leq \alpha |A|$ for all $i\in \NN$. For two colourings $\chi,\phi:A\rightarrow \NN$, we say $\phi$ is a \textit{merging} of $\chi$ if there exists $\pi \colon \NN \to \NN$ such that $\phi(i) = \pi(\chi(i))$ for every $i \in A$.
 Thus the partition defined by the colouring $\phi$ can be obtained from the partition defined by $\chi$ by merging some colour classes. 

\begin{observation}
\label{obs:merging}
 If $\chi,\,\phi:A\rightarrow \NN$ and $\phi$ is a merging of $\chi$ then any $k$-AP that is rainbow with respect to $\phi$ is also rainbow with respect to $\chi$. 
\end{observation} 

We also need the following simple lemma. 

\begin{lemma}[Colour merging] \label{lemma:merging}
   Let $0<\alpha\leq 1$ and
   $A\subseteq \NN$.  If $\chi:A\rightarrow \NN$ is $\alpha$-bounded,
   then there exist some $r\in \NN$ with $r\leq 4/\alpha$ and an $\alpha$-bounded $r$-colouring $\phi:A\rightarrow [r]$ which is a merging of~$\chi$. 
\end{lemma}
\begin{proof}
  We  apply a greedy merging procedure to merge sparse colour classes induced by $\chi$ and
  obtain a nearly balanced colouring $\phi \colon A \to [r]$ with $r:=\lfloor 4/\alpha\rfloor$.
  We initiate by setting $\phi = \chi$, and, while there exist
  two colours $i,j\in \NN$ with $|\phi^{-1}(i)|, |\phi^{-1}(j)|\leq \alpha |A| / 2$, we pick two of them arbitrarily and merge them, updating $\phi$ so that $\phi(a)=i$ for all $a\in \phi^{-1}(j)$. We repeat this
  procedure until every colour class but one contains at least $\alpha |A|/2$
  elements, so that there will be at most $2/\alpha + 1 \leq r$ colours
  left in $\phi(A)$, each defining colour classes with  less than $\alpha |A|$ elements. Upon relabeling  $\phi(A)$  with $[r]$, we obtain the desired colouring $\phi$ which is a merging of $\chi$. 
\end{proof}

\subsection{Hypergraph containers}
For a set $V$ and $1\leq t\leq |V|$, we let $\cP(V):=\{S:S\subseteq
V\}$ denote the power set of $V$ and $\binom{V}{\leq t}$ denote all
sets in $\cP(V)$ of size at most $t$. For a $k$-uniform hypergraph
$\cH=(V(\cH),E(\cH))$, we let $v(\cH):=|V(\cH)|$ and
$e(\cH):=|E(\cH)|$. For $1\leq \ell\leq k$ and a set $S\subset V(\cH)$
with $|S|=\ell$, let~$\deg(S)$ be the number of edges in $E(\cH)$ that
contain $S$. The maximum $\ell$-degree of~$\cH$, denoted $\Delta_\ell(\cH)$, is the maximum value of $\deg(S)$ over all sets $S\subseteq V(\cH)$ of size $\ell$. Finally $\cI(\cH)$ denotes the collection of independent sets in $\cH$.  

We make use of the hypergraph container method, developed
independently by Balogh, Morris and Samotij~\cite{2015.BMS}, and by
Saxton and Thomason~\cite{2015.ST}. The central insight of this method
is that, under some suitable conditions, the family of independent
sets in a hypergraph can be efficiently covered by a small collection
of \emph{containers}—vertex subsets that induce few hyperedges. Every
independent set is contained in at least one container  and the total
number of containers is substantially smaller than the number of
independent sets. This reduction allows one to reason about containers
rather than independent sets, leading to more efficient union bounds.
We give the result in the following form, which can be derived from
\cite[Theorem 2.2]{2015.BMS} by setting $\cF\subseteq \cP(V(\cH))$ to
be   the   family $\cF=\cF_{\eps'}:=\{W\subseteq V(\cH):e(\cH[W])\geq \eps' e(\cH) \mbox{ and }|W|\geq \eps' v(\cH)\}$.

\begin{theorem}[Hypergraph containers]
  \label{theorem:containers}
  For every  $k\in \NN$ and $c,\varepsilon'>0$,
  there exists $C'>0$ such that the following holds. Suppose $\cH$ is a
  $k$-uniform hypergraph and $p \in (0, 1)$ is such that, for every
  $\ell \in [k]$,
  \[
    \Delta_\ell(\cH) \leq c p^{\ell-1} \frac{e(\cH)}{v(\cH)}.
  \]
  Then there exists a family $\cS \subseteq \binom{V(\cH)}{\leq C' p v(\cH)}$ and
  functions $f \colon \cS \to \cP(V(\cH))$ and $g \colon \cI(\cH) \to \cS$
  such that:
  \begin{enumerate}[label={\rm(\Roman*)}]
    \item \label{cont:small}For every $S \in \cS$,
    we have that $
        |f(S)| < \varepsilon' v(\cH)$ or $
        e(\cH[f(S)]) < \varepsilon' e(\cH).$
    \item  \label{cont: indep}For every $I \in \cI(\cH)$, we have that $
        g(I) \subseteq I$ and $ I \setminus g(I) \subseteq
        f(g(I)).$
     \end{enumerate}
\end{theorem}

\subsection{Szemerédi's theorem}
We will also need to use versions of Szemerédi's
classical theorem, which states that sets of positive density in $[n]$ contain arithmetic progressions of arbitrary length.  For $\alpha>0$ and $k\in \NN$, we 
say that a set
$A$ is \emph{$(\alpha,k)$-Szemer\'edi} (or  $\aksz$ for short)  if 
every $B \subseteq A$ with $|B| \geq \alpha |A|$ contains a
  $k$-AP. The following theorem, which transfers Szemer\'edi's theorem to sparse random sets, was proven by Schacht~\cite{2016.Schacht} and independently by Conlon and Gowers~\cite{2016.CG} (with slightly weaker probability bounds); see also \cite[Corollary 4.1]{2015.BMS}. 

\begin{theorem}[Random Szemerédi]
  \label{theorem:sparse_szemeredi}
  For every integer $k>0$ and $\alpha > 0$, there exist constants $C,
  c > 0$ such that for $p > Cn^{-1/(k-1)}$, the
  random set $[n]_p$ is $\aksz$ with probability at
  least $1 - e^{-cnp}$.
\end{theorem}

 Finally we need the following
supersaturated version of the Szemerédi's Theorem, first proven by Varnavides \cite{1959.V} (see also \cite[Lemma 4.2]{2015.BMS}). 

\begin{theorem}[Supersaturated Szemerédi]
  \label{theorem:szemeredi}
  For every integer $k>0$ and $\delta > 0$, there exists $\epsilon >
  0$ such that, for sufficiently large $n$, any $A \subseteq [n]$ with
  $|A| \geq \delta n$ contains at least $\epsilon n^2$ $k$-APs.
\end{theorem}

\section{Proof of Theorem \ref{thm:random_canonical_vdw}}
\label{sec:outline}

As noted in the introduction, to establish Theorem \ref{thm:random_canonical_vdw}, it suffices to prove the \ref{can random 1}-statement, which we restate here with quantitative bounds on the error probability.   This slightly strengthened version is needed   for
our proof of Corollary \ref{cor:sparse_vdw}.

\begin{theorem}[Random canonical van der Waerden]
  \label{thm:quant_sparse_canonical_vdw}
  For every integer $3 \leq k\in \NN$, there exist $C, c > 0$ such that
  for $p > Cn^{-1/(k-1)}$  the random set $[n]_p$
  is $\cankvdW$ with probability at least $1 - e^{-cnp}$.
\end{theorem}

Erdős and Graham~\cite{1980.EG}
proved their canonical version of the van der Waerden theorem making
use of Szemerédi's theorem~\cite{szemeredi75}, while an elementary proof
was given by Prömel and Rödl~\cite{promel86:_gallai_witts}.  We
follow the Erdős--Graham approach: 
the proof of Theorem \ref{thm:quant_sparse_canonical_vdw} starts
out with a basic dichotomy over a given colouring of $[n]_p$. If such
a colouring has a colour with positive density, we are able to apply
Szemerédi's theorem for random sets (Theorem
\ref{theorem:sparse_szemeredi}) to find a monochromatic $k$-AP.  If,
on the other hand, all colours are sparse, it turns out that we can
find a rainbow $k$-AP. For $\alpha>0$ and $k\in \NN$, we say a set $A\subseteq \NN$ is $\akrb$ if every $\alpha$-bounded colouring of $A$ results in a rainbow $k$-AP.

\begin{proposition}[Rainbow $k$-APs in bounded colourings]
  \label{prop:bounded_sparse_vdw}
  For every $3 \leq k \in \NN$, there exist  $C,\,c,\,\alpha>0$ such that  for sufficiently large $n$ and  $p > Cn^{-1/(k-1)}$,  the random set $[n]_p$ is $\akrb$ with probability at least $1-e^{-cnp}$. 
\end{proposition}
\begin{proof}[Proof of Theorem \ref{thm:quant_sparse_canonical_vdw} assuming Proposition \ref{prop:bounded_sparse_vdw}]
 Fixing $3\leq k\in \NN$ and applying
  Proposition~\ref{prop:bounded_sparse_vdw} we obtain
  constants $C_1$, $c_1$ and $\alpha$. Then, applying
  Theorem~\ref{theorem:sparse_szemeredi} with $k$ and $\alpha$ we get
  constants $C_2$ and~$c_2$. Therefore taking $C=\max\{C_1,C_2\}$ and $c=\min\{c_1,c_2\}/2$, we have that if $p>Cn^{-1/(k-1)}$ then with probability at least $1-e^{-cnp}$, the random set $[n]_p$ is simultaneously $\aksz$ and $\akrb$. The conclusion now follows as for any colouring $\chi:A\rightarrow \NN$, if there is some colour $i\in \NN$ with $|\chi^{-1}(i)|\geq \alpha |A|$, then there is a monochromatic $k$-AP in colour $i$ using that $A$ is $\aksz$. If, on the other hand, the colouring $\chi$ is $\alpha$-bounded, then there is a rainbow $k$-AP on account of $A$ being $\akrb$.   
\end{proof}

Thus it remains to prove Proposition \ref{prop:bounded_sparse_vdw}.
 In order to do this, our proof follows the scheme of Nenadov and Steger~\cite{2016.NS}, who used the hypergraph container theorem (Theorem \ref{theorem:containers}) to prove the \ref{random 1}-statement of Theorem \ref{thm:random_vdw}. In Section
\ref{sec:rainbow_hypergraph} we therefore define 
 a suitable \emph{rainbow hypergraph} that encodes 
rainbow $k$-APs and apply the container theorem to this hypergraph.  Next,
in Section \ref{sec:supersaturation}, we apply a supersaturation result for 
 rainbow $k$-APs   to prove a key lemma (Lemma \ref{lemma:few_sols_few_colors}), which provides  structural information about our containers. We regard this key lemma, as well as  its exploitation, as  the most interesting features  of our proof. Finally, in
 Section~\ref{sec:pieces} we conclude our desired result by appealing to a union bound over all containers. 

\subsection{Rainbow hypergraph}
\label{sec:rainbow_hypergraph}
We define the \emph{rainbow hypergraph}
$\cR = \cR(n, k, r)$ to be the $k$-uniform hypergraph with vertex set consisting of $r$ copies of
$[n]$, one for every possible colour $\omega\in[r]$, and edge set formed by all possible rainbow $k$-APs.
Formally, 
$V(\cR) =  [r]\times [n]$ and 
\[
  E(\cR) = \left\{\left\{(\omega_1,a_1), \dots, (\omega_k,a_k)\right\} \in \binom{V(\cR)}k  \colon (a_1, \dots, a_k) \text { forms a }
  k\text{-AP and }
\omega_i \neq \omega_j \, \forall i \neq j
\right\}.
\]
(The interested reader is referred to~\cite{lin22:_integ,li22:_integ}
for a different application of this hypergraph.) 

\begin{lemma}
  \label{lemma:regularity_rainbow}
  For integers $r \geq k \geq 3$, there exists $c > 0$ such
  that, for $1 \leq \ell \leq k$, the rainbow hypergraph $\cR =
  \cR(n, k, r)$ satisfies 
  \[
    \Delta_\ell(\cR) \leq c n^{-(\ell-1)/(k-1)} \frac{e(\cR)}{v(\cR)}.
  \]
\end{lemma}
\begin{proof}
  Let us first lower bound $e(\cR)$. For $a, d \in
  [n]$, the $k$-AP $\{a, a+d, \dots, a+(k-1)d \}$ certainly lies in
  $[n]$ if $a \leq n/k$ and $d \leq n/k$. This gives a lower bound of
  at least $(n/k)^2$ different $k$-APs contained in $[n]$. As $r \geq k$, every $k$-AP in $[n]$ gives rise
  to at least one rainbow $k$-AP, which gives
$    e(\cR) \geq \left(n/k\right)^2.$

  A given integer $a$ in $[n]$ belongs to at most $kn$ different $k$-APs
  contained in $[n]$, accounting for its position in a $k$-AP and the
  value of the common difference. Therefore  a given vertex $(\omega,a)\in V(\cR)$ is contained in at most
  \[
    \Delta_1(\cR) \leq k n r^{k-1} \leq k^3 r^{k} \frac{n^2}{k^2 nr} \leq
    k^3 r^{k}\frac{e(\cR )}{v(\cR )}
  \]
  edges of $\cR$. 
Given two different integers in $[n]$ there are at most $k^2$
  different $k$-APs containing them, accounting for their positions in the $k$-AP. Therefore, for $2 \leq \ell \leq k$
  it holds that
  \[
    \Delta_\ell(\cR) \leq k^2r^{k-2} \leq k^4 r^{k-1} \frac{e(\cR)}{n v(\cR)} \leq 
    k^4 r^{k-1} n^{-(\ell-1)/(k-1)}\frac{e(\cR)}{v(\cR)}.
  \]
  Setting $c: = \max(k^3r^{k}, k^4r^{k-1}) = k^3r^{k}$ gives the result.
\end{proof}

With Lemma \ref{lemma:regularity_rainbow} in hand we can apply Theorem
\ref{theorem:containers} to obtain the following.

\begin{lemma}[Containers for rainbow  hypergraphs]
  \label{theo:rainbow_containers}
  For  integers $r \geq k \geq 3$ and $0<\eps<1/2$, there
  exists a constant $C> 0$ such that the following holds. For all $n
  \in \bbn$, taking $\cR=\cR(n,k,r)$ and $n':=\lfloor Cn^{1-1/(k-1)}\rfloor$, we have that there exists a collection of \emph{fingerprints} $\cS\subseteq \binom{V(\cR)}{\leq n'}$, a collection of \emph{containers} $\cW\subseteq \cP(V(\cR))$ and a function $f \colon \cS\to \cW$ such that:
  \begin{enumerate}[label={\rm(\roman*)}]
       \item \label{rb cont: almost indep} Every $W \in \cW$ satisfies $e(\cR[W])
    < \varepsilon n^2$.
  \item \label{rb cont: indep}For every independent set $I\in \cI(\cR)$, there is some $S\in \cS$ such that $ S\subseteq I$ and $I\setminus
    S \subseteq f(S)$.
  \end{enumerate}
\end{lemma}
\begin{proof}
    We apply Theorem \ref{theorem:containers} with   $0<\eps':=\eps/(k^2 r^k)$ and $c$ the constant output by Lemma \ref{lemma:regularity_rainbow} with input $r$ and $k$. Thus Theorem \ref{theorem:containers} outputs a constant $C'>0$ and we fix $C:=rC'$. 
     Furthermore, fixing  $p:= n^{-1/(k-1)}$  and $\cH=\cR$, 
Lemma \ref{lemma:regularity_rainbow} and Theorem
\ref{theorem:containers} imply the existence of a set $\cS\subseteq
\binom{V(\cR)}{\leq n'}$ and functions $f:\cS\rightarrow \cP(V(\cR))$
and $g: \cI(\cR)\rightarrow \cS$ as in the conclusion of Theorem
\ref{theorem:containers}. We  define $\cW:=\{f(S):S\in \cS\}\subseteq \cP(V(\cR))$.

Let us check conditions \ref{rb cont: almost indep} and \ref{rb cont: indep} for $\cW$ and  $f:\cS\rightarrow \cW$.  For \ref{rb cont: almost indep}, we have that for each $W\in \cW$ condition \ref{cont:small} of Lemma \ref{theorem:containers} gives that $|W|< \eps'v(\cR)$  or $e(\cR[W])<\eps'e(\cR)$. Therefore,
\[
e(\cR[W]) < \max\{(\eps' v(\cR))^2 k^2 r^{k-2}, \eps'e(\cR)\} \leq \max\{(\eps'rn)^2k^2r^{k-2}, \eps'n^2r^k\}\leq \eps n^2,
\]
using that a pair of integers is contained in at most $k^2$ $k$-APs and that $\eps':=\eps/(k^2 r^k)$. For condition~\ref{rb cont: indep}, if $I\in \cI(\cR)$ is an independent set in $\cR$, we take $S=g(I)\subseteq I$. Then Theorem \ref{theorem:containers} guarantees that $|S|\leq C'n^{-1/(k-1)}v(\cR)\leq Cn^{1-1/(k-1)}$ and $I\setminus S\subseteq f(S)$.
\end{proof}

\subsection{Supersaturation for rainbow \texorpdfstring{$k$}{k}-APs}
\label{sec:supersaturation}
Supersaturation  is the phenomenon that above an extremal threshold for some substructure of interest, one in fact has many copies of that substructure, as in  Theorem \ref{theorem:szemeredi} for $k$-APs.   One can also show supersaturation for coloured copies in Ramsey settings. For example, by considering the largest colour class one can easily infer from Theorem \ref{theorem:szemeredi} that for any fixed  $r$ and $n$ sufficiently large, an $r$-colouring of any dense subset of $[n]$ induces many \textit{monochromatic } $k$-APs. Our next lemma gives such a result for rainbow $k$-APs in \textit{bounded} colourings.

\begin{lemma}[Rainbow $k$-AP supersaturation]
  \label{lemma:rainbow_van_der_waerden}
  For every $3\leq k\in \NN$ and $\delta > 0$, there exist  $\epsilon,\, \beta>0$ such that the following holds for
  large enough $n$.
  Every $\beta$-bounded colouring of a subset $A \subseteq [n]$ with
  $|A| \geq \delta n$ contains at least $\varepsilon n^2$ rainbow
  $k$-APs.
\end{lemma}
\begin{proof}
  Let $\epsilon_1$ be the constant provided by Theorem
  \ref{theorem:szemeredi} for sets of density $\delta$ and  fix $\beta :=
  \epsilon_1/2k^2$  and $\eps:=\epsilon_1/2$. Taking $A\subseteq \NN$ with $|A|\geq \delta n$, and $\phi:A\rightarrow [n]$ to be some $\beta$-bounded colouring of $A$, we upper bound the number of $k$-APs in $A$ with two elements of the same colour
  according to $\varphi$. There are at most $|A|\leq n$ ways to choose
  some $a_1\in A$ and at most $\beta |A|\leq \beta n$ ways to choose
  $a_2\in A$ with $\phi(a_1)=\phi(a_2)$. Given two elements $a_1, a_2
  \in A$, the number of $k$-APs containing both $a_1$ and~$a_2$ is at most~$k^2$.
  Hence, there are at most  $\beta k^2n^2$  $k$-APs using a repeated colour. By Theorem \ref{theorem:szemeredi}, there are at least~$\epsilon_1
  n^2$ $k$-APs in $A$ and so the total number of rainbow $k$-APs
  is at least
 $    \epsilon_1 n^2 - \beta k^2 n^2 \geq \eps n^2.$
\end{proof}

In Ramsey settings with a bounded number of colours, as considered by Nenadov and Steger \cite{2016.NS}, supersaturation results immediately give useful information about containers. Indeed, as a container induces $o(e(\cH))$ edges of the hypergraph (which correspond to monochromatic $k$-APs in that setting), one can conclude that the container `misses many integers', that is, 
the projection of the container onto $[n]$ must be sparse. In our
setting, as Lemma \ref{lemma:rainbow_van_der_waerden} only applies to
\textit{bounded} colourings, it is not immediate how to deduce
structural information about containers. Nonetheless, we are able to
do so in the form of Lemma~\ref{lemma:few_sols_few_colors} below.
This key lemma implies that any container either misses many integers, or there is some large subset of integers that induces few colours in the container. 

To make this precise, we introduce some notation.
For $\cR=\cR(n,k,r)$ and any vertex subset $W\subseteq V(\cR)$ and integer $x\in [n]$, 
we let 
\[W_x:=\{\omega\in [r]:(\omega,x)\in W\}.\] 
In words, $W_x$ is the set of colours in which $x$ appears in $W$. We also take $\pi_{[n]}(W):=
\left\{x \in [n] \colon W_x \neq \varnothing \right\}$ to be the projection of $W$ onto the integers $[n]$. 

\begin{lemma}[Structure of containers]
  \label{lemma:few_sols_few_colors}
  Given $3\leq k\in \NN$, there exists $M\in \NN$ and $\eps>0$ such that  for every
  $r,\,n\in \NN$ the following holds for $\cR=\cR(n,k,r)$.  We have that if $W\subseteq V(\cR)$ is such that $e(\cR[W]) <
  \varepsilon n^2$ and  $A:=\pi_{[n]}(W)$ has size $|A| \geq 3n/4$, then there is some subset of integers $B \subseteq A$ and a subset of colours $\Omega\subseteq [r]$ such that 
  \begin{equation}
\label{eq:conditions_container}
|B| \geq n/4, \quad |\Omega|\leq M \quad  \text{and} \quad  W_b\subseteq \Omega \mbox{ for all } b\in B.
  \end{equation}
\end{lemma}
Note that the constant $M$ depends only on $k$ and that the result holds for any number of colours $r$.  
\begin{proof}[Proof of Lemma~\ref{lemma:few_sols_few_colors}]
 Fix $\delta:=1/4$, let $\eps_1>0$ be the constant provided by Theorem
 \ref{theorem:szemeredi} with input $\delta$, let $\eps_2>0$ and
 $\beta>0$ be the constants provided by Lemma
 \ref{lemma:rainbow_van_der_waerden} with input $\delta$ and fix
 $M:=\lceil 4k/\beta \rceil$ and  $\eps:=\min\{\eps_1,\eps_2\}$.

 Now let~$W$ be as in the statement of the lemma.  We start out by showing that most integers $x\in A$ induce few colours. Indeed, consider the set 
  \[
    D := \left\{ x \in A \colon |W_x| \geq k \right\}.
  \]
  Since $|W_x| \geq k$ for every $x \in D$, a $k$-AP using integers in $D$
  induces at least one rainbow $k$-AP and therefore
  an edge in $\cR[W]$. Hence, the number of $k$-APs in $D$ is less than  $e(\cR[W]) <
    \varepsilon n^2$ 
   and as $\eps\leq \eps_1$,   Theorem \ref{theorem:szemeredi} implies 
  that $|D| \leq n/4$. We thus have 
  that
  $
    A' := A \setminus D = \left\{ x \in A \colon |W_x| < k \right\}
  $
  satisfies $|A'| \geq n/2$.

  We now define the set of colours $\Omega$ and $B\subseteq A'$
  as 
  \[
    \Omega := \left\{ \omega \in [r] \colon
      |W\cap(\{\omega\}\times A')|\geq\beta n/4
    \right\} \quad \mbox{and} \quad B:=\{b\in A':W_b\subseteq \Omega\}.
  \]
 The third part of \eqref{eq:conditions_container} is thus satisfied by definition and it only remains to bound the sizes of $\Omega$ and $B$.   
  Firstly, as $|W_x|<k$ for all $x\in A'$ and $|A'|\leq n$, by double counting we have that 
  \[
   |\Omega| \beta n /4 \leq |W\cap (\Omega \times A')|  <n k,
  \]
 and so $|\Omega| < 4 k/\beta\leq M$ as required.  Now consider the set $A'':=A'\setminus B$ 
  and define a colouring $\psi \colon A'' \to [r]$ such that $\psi(x) \not \in \Omega$ and
  $\psi(x) \in W_x$ for all $x \in A''$, noting that this is possible due to the definition of~$B$.  Also note that every colour will appear at most $\beta
  n/4$ times as~$\psi(x)$ ($x\in A''$) due to the definition of~$\Omega$.
  Suppose that $|A''| \geq n/4$. In that case, $\psi$ is a $\beta$-bounded
  colouring of $A''$ and Lemma $\ref{lemma:rainbow_van_der_waerden}$ ensures the
  existence of $\varepsilon_2 n^2>\eps n^2$ rainbow $k$-APs, contradicting that $e(\cR[W]) <\eps n^2$.
  Therefore, we 
  conclude that $|A''| < n/4$ and $B$ has size $|B|= |A'|-|A''|\geq n/4$, which finishes the proof. 
\end{proof}

\subsection{Putting everything together}
\label{sec:pieces}
We are now ready to prove Proposition \ref{prop:bounded_sparse_vdw} and thus complete the proof of Theorems \ref{thm:quant_sparse_canonical_vdw} and \ref{thm:random_canonical_vdw}. 
\begin{proof}[Proof of Proposition \ref{prop:bounded_sparse_vdw}]
  Let $M\in \NN$ and $\eps>0$ be the constants provided by Lemma \ref{lemma:few_sols_few_colors} 
  and 
  set $\alpha := 1/(32M)$ and $r: = \lceil 4/\alpha \rceil$.   Let
  $C_1>0$ be output by Lemma  \ref{theo:rainbow_containers} with
  inputs $r$, $k$ and~$\eps$. We will take $C>0$ sufficiently large and $c>0$ sufficiently small to satisfy any constraints we encounter in the proof. 
  We also take $n$ large and
   fix $n':=\lceil{C_1n^{1-1/(k-1)}}\rceil$, and consider the rainbow hypergraph $\cR=\cR(n,k,r)$, the set of fingerprints  $\cS:=\binom{V(\cR)}{\leq n'}$, 
  the containers $\cW\subseteq \cP(V(\cR))$  and the function $f:\cS\rightarrow \cW$ obtained by applying Lemma~\ref{theo:rainbow_containers} to~$\cR$.

Now let $\cZ\subseteq \cP([n])$ be the family of all the sets of
integers that are \textit{not}~$\akrb$. To each set  $Z\in \cZ$, we
associate an independent set $I(Z)\in \cI(\cR)$ as follows. By
definition of the  set $Z\subseteq [n]$ not being $\akrb$,  there is
some $\alpha$-bounded colouring $\chi:Z\rightarrow \NN$ which avoids
rainbow $k$-APs. Take  $\phi:Z\rightarrow [r]$ to be an
$\alpha$-bounded colouring which is a merging of $\chi$ as in Lemma
\ref{lemma:merging}, and note that~$\phi$ also induces no rainbow
$k$-APs by Observation~\ref{obs:merging}.  Fix one such~$\phi$ for
each $Z\in\cZ$.
Finally take $I=I(Z)=\{(\phi(z),z):z\in Z\}\subseteq V(
\cR)$ to be the natural mapping of $Z$ coloured by $\phi$ to $V(\cR)$
and note that $\pi_{[n]}(I(Z))=Z$ and that~$I$ is  independent
in~$\cR$.   

For a fingerprint $S\in \cS$, we say a set $Z\subseteq [n]$ is \textit{compatible} with $S$ if $Z\in \cZ$, $S\subseteq I(Z)$  and $I(Z)\setminus S\subseteq f(S)$.
By our container lemma, Lemma \ref{theo:rainbow_containers}, we have that for every $Z\in \cZ$, there is some $S\in \cS$ such that $Z$ is compatible with $S$. Therefore by a union bound, 
  \begin{equation}
    \label{eq:bound_prob_containers}
      \Pr([n]_p \text{ is not } \akrb) =\Pr([n]_p\in \cZ) 
      \leq \sum_{ S \in \cS} \Pr\big([n]_p \text{ compatible with } S\big).
  \end{equation}
Now for a fixed $S\in \cS$, we bound the probability that $[n]_p$ is
compatible with $S$. Let $Y=Y(S):=\pi_{[n]}(S)$, let $\overline
Y=\overline Y(S):=[n]\setminus Y$ and let $A=A(S):=\pi_{[n]}(W)$, where
  $W:= f(S)$ is the container associated with~$S$. If $[n]_p$ is compatible with $S$, then in particular  we have that $Y\subseteq \pi_{[n]}(I([n]_p))= [n]_p$ and $[n]_p\cap \overline Y\subseteq \pi_{[n]}(I([n]_p)\setminus S)\subseteq A$.  Moreover note that the events $Y\subseteq [n]_p$ and $[n]_p\cap \overline{Y}\subseteq A$  are independent.  
  Therefore if $|A| \leq 3n/4$, then 
  \begin{equation}
   \label{eq:bound_small_container}
   \Pr\big([n]_p \text{ compatible with } S \big) \leq \Pr(Y
	\subseteq [n]_p) \Pr([n]_p \cap \overline Y \subseteq A) 
    \leq p^{|Y|} (1-p)^{n-|Y\cup A|} 
     \leq p^{|Y|} e^{-np/8},
         \end{equation}
  using  that  $|Y \cup A| \leq 3n/4+n'\leq  7n/8$ for large enough $n$.

  On the other hand,  if $|A| \geq 3n/4$, then as $e(\cR[W])<\eps n^2$ by Lemma \ref{theo:rainbow_containers} \ref{rb cont: almost indep},  we can apply Lemma
  \ref{lemma:few_sols_few_colors} to obtain a subset $B \subseteq A$ and a set of colours $\Omega\subseteq [r]$ satisfying
  \eqref{eq:conditions_container}. Taking $B':=B\cap \overline{Y}$,   if $Z\subseteq [n]$ is
  compatible with $S$,  we have that 
  \begin{equation}
    \label{eq:5}
    |Z \cap B'|= |(I(Z)\setminus S)\cap ([r]\times B)|\leq |I(Z)\cap (\Omega \times[n])|,
  \end{equation}
  using that $I(Z)\setminus S\subseteq W$ and the fact that  $W_b\subseteq \Omega$ for all $b\in B$. Therefore, taking $\varphi:Z\rightarrow [r]$ to be  the $\alpha$-bounded colouring
  used to generate $I(Z)$, we get that 
  \begin{equation} \label{eq:Z cap B bound}
    |Z \cap B'|\leq  \sum_{\omega\in \Omega} |I(Z)\cap (\{\omega\}\times [n])|  = \sum_{\omega \in \Omega}\big|\varphi^{-1}(
    \omega
    ) \cap Z\big| 
    \leq \sum_{\omega \in \Omega} \alpha |Z| \leq M \alpha |Z|,
  \end{equation}
  because $\varphi$ is $\alpha$-bounded and $|\Omega|\leq M$.  Thus if
  $|Z|\leq 2 pn$ then $|Z\cap B'| \leq np/16$ because of the
  definition of~$\alpha$.  Hence
  \begin{align}
    \Pr\big(\big|[n]_p\big|\leq2pn\text{ and }[n]_p\text{ is compatible with
    }S\big)
    &\leq\Pr\big(Y\subset[n]_p\text{ and }\big|[n]_p\cap B'\big|\leq
      np/16)\nonumber\\
    &=\Pr\big(Y\subset[n]_p)\Pr(\big|[n]_p\cap B'\big|\leq np/16).
    \label{eq:1}
  \end{align}
  Take~$C$ large enough with respect to~$C_1$ so that
  $|Y|\leq n'\leq np/2$.  Then
  \begin{align}
    \Pr\big(\big|[n]_p\big|>2pn\text{ and }[n]_p\text{ is compatible with
    }S\big)
    &\leq\Pr\big(Y\subset[n]_p\text{ and
      }\big|[n]_p\cap\overline{Y}\big|>3np/2)\nonumber\\ 
    &=\Pr\big(Y\subset[n]_p)\Pr(\big|[n]_p\cap\overline Y\big|>3np/2).
    \label{eq:2}
  \end{align}
  Now note that $|B'|\geq |B|-|Y|\geq n/8$ for~$n$ sufficiently large,
  and hence $\mathbb{E}(|[n]_p\cap B'|)\geq pn/8$.  Also,
  $\mathbb{E}(|[n]_p\cap\overline Y|)\leq np$.  Combining~\eqref{eq:1}
  and~\eqref{eq:2} and applying Chernoff's inequality (see, for example, Appendix~A
  in~\cite{2016.AS}), we have
  \begin{align}
      \Pr\big([n]_p \text{ compatible with } S\big) 
      & \leq \Pr\big(Y \subseteq [n]_p\big) 
      \Big(\Pr\big(\big|[n]_p \cap B'\big| \leq np/16\big) + \Pr\big(\big|[n]_p\cap \overline Y\big|
      >3np/2\big)\Big) \nonumber \\
      &\leq p^{|Y|} e^{-c_1 np},
      \label{eq:bound_large_container}
  \end{align}
  for some constant $c_1>0$, which we may suppose
  satisfies~$c_1\leq1/8$.  Using either
  \eqref{eq:bound_small_container} or \eqref{eq:bound_large_container}
  depending on the size of $A=A(S)$, we can bound the sum in
  \eqref{eq:bound_prob_containers} to obtain
  \[
      \Pr([n]_p \text{ is not } \akrb) \leq \sum_{S \in \cS}
      p^{|Y(S)|} e^{-c_1np}.
  \]
  Recall that every $S\in \cS$ has size at most $n'$  and so $|Y(S)|\leq n'$ also.
  Moreover, for any set $Y'\subset [n]$ of size $0\leq y\leq
  n'$, 
  there are at most $(2^r)^y$  choices of $S\in \cS$ with $Y(S)=Y'$. Therefore, splitting
  according to the size of $Y(S)=\pi_{[n]}(S)$, we have
  \begin{equation}  \label{eq:8}
        \Pr([n]_p \text{ is not } \akrb)  \leq \sum_{y =
        0}^{n'} \binom{n}{y} 2^{ry} p^y e^{-c_1np}
         \leq 2\sum_{y = 1}^{n'} \left(\frac{C_2 n p}{y}\right)^y
      e^{-c_1np},
  \end{equation}
  for some large enough $C_2 = C_2(k) > 0$, where the factor of 2 is to account for the $y=0$ term.
  The function $h(x)=(t/x)^x$ is increasing for $x < t/e$. Therefore,
  provided that $p \geq  C n^{-1/(k-1)}$ for a large enough constant $C = C(k)$, 
  the inner term in the last sum is maximised when $y$
  is maximum. Fixing $q:=n^{-1/(k-1)}$, this implies that the right-hand side of \eqref{eq:8} is at most
  \begin{equation}
       2n'\left(\frac{C_2 p}{C_1q}\right)^{2C_1 qn} e^{-c_1 pn} \leq n \,
    \exp{\left(-pn\left(c_1-2C_1\frac{q}{p}\log
          {\left(\frac{C_2p}{C_1q}\right)}\right)\right)}
    \leq  e^{-cpn}
  \end{equation}      
  for all large enough~$n$, as long as $c>0$ is a sufficiently small
  constant and $C>0$ is a sufficiently large constant, because
  $\lim_{x\to 0}{x\log{1/x}}=0$.
\end{proof}

\section{Proof of Corollary \ref{cor:sparse_vdw}} \label{sec:cor}

We prove Corollary \ref{cor:sparse_vdw} 
by taking $[n]_p$ with an appropriate probability $p=p(n)$ and checking that with
positive probability (albeit tending to zero with $n$) it satisfies both desired
properties. Recall  that a cycle~$\cC$ of length $2\leq\ell\in \NN$
in a hypergraph is a subgraph with distinct vertices
$v_1,\ldots,v_\ell\in V(\cC)$ and an ordering of
$E(\cC)$ as $e_1,\ldots,e_\ell$ such that $v_i\in e_i\cap e_{i+1}$ for
$i=1,\ldots,\ell=|E(\cC)|$ (with $e_{\ell+1}=e_1$).  We call the vertices
$v_1,\dots,v_\ell\in V(\cC)$ the \textit{linking} vertices of the
cycle~$\cC$ and let $V_L(\cC):=\{v_1,\ldots,v_\ell\}$.
  A cycle is \emph{minimal} if it contains no subset of
vertices and edges forming a cycle of smaller length. Minimal cycles satisfy the
following.
\begin{lemma}
  \label{lemma:minimal_cycles}
  For $\ell \geq 3$, a minimal cycle of $\cC$ of length $\ell$ in a $k$-uniform
  hypergraph has  $v(\cC)=(k-1)\ell$.
\end{lemma}
\begin{proof}
Given a hypergraph $\cH$, let the bipartite
\textit{incidence graph}~$G_I(\cH)$ for~$\cH$ have vertex set
$V(G_I(\cH)):=V(\cH)\cup E(\cH)$, with a pair $(v,e)\in
V(\cH)\times E(\cH)$ being an edge of $G_I(\cH)$ if and only if $v\in
e$. Note that  cycles of length $\ell\geq 2$ in a hypergraph $\cH$
with an identified linking vertex set~$V_L$ are in one to one
correspondence with cycles of length $2\ell$ in $G_I(\cH)$.
Take~$\cC$ to be a minimal cycle of length~$\ell$ with linking vertex set
$V_L=V_L(\cC)$, and let~$C$ be the corresponding length $2\ell$ cycle in $G_I:=G_I(\cC)$. Then:
  \begin{itemize}
    \item Each $v\in V_L$ has degree 2 in $G_I$. Indeed if this were not the case then in $G_I[V_L\cup E(\cC)]$ the edge incident to $v$ and not in the cycle $C$ will be a chord in  $C$, giving  smaller even cycles $C',\,C''$   which correspond to a smaller cycles $\cC',\,\cC''\subseteq \cC$ contradicting the minimality of $\cC$.
    \item Each $u\in V(\cC)\setminus V_L$ has degree 1 in
      $G_I$. Indeed, if this is not the case, then in $G':=G_I[V_L\cup
      \{u\}\cup E(\cC)] $ the vertex $u$ and two of its incident edges
      form a subdivided chord of $C$ giving two cycles in $G'$ which
      contain $u$, at least one of which, as we assume $\ell\geq3$, has length less than~$2\ell$ and so contradicts the minimality of $\cC$.
  \end{itemize}
Therefore counting the edges from either side of $G_I$ we have that
$k\ell=2\ell+|V(\cC)\setminus V_L|$, which in turn implies that
$k\ell=\ell+|V(\cC)|$, whence $|V(\cC)|=(k-1)\ell$,
as required.  
\end{proof}

\begin{proof}[Proof of Corollary \ref{cor:sparse_vdw}]
  Fix~$k\geq3$ and let~$C$ and~$c$ be the corresponding positive
  constants given by Theorem~\ref{thm:quant_sparse_canonical_vdw}.
  Let $p = Cn^{-1/(k-1)}$ and let $\cH := \hkap([n]_p)$ be the random
  $k$-uniform hypergraph whose hyperedges are the $k$-APs in~$[n]_p$.
  Let~$g\geq3$ be given.  We prove that~$\cH$ is as required with
  positive probability as long as~$n$ is large enough.

  Let us first upper bound the probability that $\cH$ contains a given minimal
  cycle $\cC$ of length $2\leq \ell< g$. If $\ell = 2$, since $\cC$ spans at least $k+1$
  vertices, we see that
  \begin{equation}
    \label{eq:lower_bound_2_cycle}
    \Pr(\cC \not\subset \cH) \geq 1-p^{k+1}.
  \end{equation}
  If $\ell \geq 3$, the cycle $\cC$ spans  $(k-1)\ell$ vertices on account
  of Lemma \ref{lemma:minimal_cycles}, in which case
  \begin{equation}
    \label{eq:lower_bound_l_cycle}
    \Pr(\cC \not\subset \cH)=1-p^{(k-1)\ell}.
  \end{equation}

  Let $\Theta_{\ell}$ be the family of all minimal cycles of length $\ell$ in
  $\hkap([n])$, and $\Theta = \bigcup_{2 \leq \ell < g} \Theta_{\ell}$ the family
  of all minimal cycles of length smaller than $g$ spanned by $k$-APs in $[n]$.
  By definition, we have that
  \[
    \Pr(\cH \text { has girth at least } g) \geq
    \Pr\Big(\bigcap_{\cC \in \Theta}\{\cC \not \subset \cH\}\Big).
  \]
  Since, for all $\cC \in \Theta$, not containing $\cC$ is a decreasing event, we may
  first apply Harris's inequality (see, for example, Chapter 6 in \cite{2016.AS})
   and then~\eqref{eq:lower_bound_2_cycle} and
   \eqref{eq:lower_bound_l_cycle} to obtain 
   \[
    \Pr \Big(\bigcap_{\cC \in \Theta} \{\cC \not \subset \cH\}\Big) 
    \geq \prod_{\cC \in \Theta} \Pr (\cC \not \subset \cH) 
    = \prod_{2 \leq \ell < g} \prod_{\cC \in \Theta_{\ell}} \Pr(\cC \not \subset \cH)
    \geq (1-p^{k+1})^{|\Theta_2|} \prod_{3 \leq \ell < g} (1-p^{(k-1)
    \ell})^{|\Theta_\ell|}.
  \]
  Fixing the linking vertices   of the cycle and upper bounding the number of $k$-APs
  that contain them, we  see $|\Theta_{\ell}| \leq
  n^{\ell}k^{2\ell}$. Using this and the fact that
  $1-x \geq e^{-2x}$ for $x \leq 1/2$, we may upper bound the previous term to
  obtain
  \[
    \Pr \Big(\bigcap_{\cC \in \Theta} \{\cC \not \subset \cH\}\Big) \geq 
    \exp\Big(-2p^{k+1}n^{2}k^4 - 2\sum_{3 \leq \ell < g} p^{(k-1)\ell}
    n^{\ell}k^{2\ell} \Big) \geq \exp(-C_1 n^{(k-3)/(k-1)}),
  \]
  for some constant $C_1 > 0$ depending on $k$ and $n$ large enough in
  terms of~$k$, $\ell$ and $g$.  On the other hand,
  Theorem~\ref{thm:quant_sparse_canonical_vdw} gives 
  \[
    \Pr([n]_p\text{ is not
    }\cankvdW)\leq\exp(-cnp)\leq\exp(-cCn^{(k-2)/(k-1)}). 
  \]
  Comparing both bounds, we obtain that
  \[
    \Pr\big([n]_p\text{ is not }\cankvdW\big) <
    \Pr\big(\hkap([n]_p) \text { has girth at least $g$}\big)
  \]
  for $n$ large, and hence a set that is $\cankvdW$ and
  whose $k$-AP hypergraph has girth at least~$g$ exists.
\end{proof}

\providecommand{\bysame}{\leavevmode\hbox to3em{\hrulefill}\thinspace}
\providecommand{\MR}{\relax\ifhmode\unskip\space\fi MR }
\providecommand{\arXiv}{\relax\ifhmode\unskip\space\fi arXiv }
\providecommand{\MRhref}[2]{%
  \href{http://www.ams.org/mathscinet-getitem?mr=#1}{#2}
}
\renewcommand{\MR}[1]{%
  \href{http://www.ams.org/mathscinet-getitem?mr=#1}{MR~#1}
}
\renewcommand{\arXiv}[1]{%
  Available as \href{https://arxiv.org/abs/#1}{arXiv:#1}.
}
\providecommand{\href}[2]{#2}

\endgroup
\end{document}